\documentclass[leqno,12pt]{article}

\usepackage{amsfonts}
\usepackage{amsmath}
\usepackage{amssymb}
\usepackage{amsthm}
\usepackage{amscd}
\usepackage{verbatim}
\usepackage{graphics}
\usepackage[all,2cell]{xypic}

\theoremstyle{plain}  
\newtheorem{thm}{Theorem}[section]
\newtheorem{con}[thm]{Conjecture}
\newtheorem{cor}[thm]{Corollary}
\newtheorem{lem}[thm]{Lemma}
\newtheorem{prop}[thm]{Proposition}

\newtheorem*{main}{Theorem}

\theoremstyle{definition}

\newtheorem{df}[thm]{Definition}
\newtheorem{ex}[thm]{Example}

\newtheorem{his}[thm]{Historical Note}

\newtheorem{ob}[thm]{Observations}
\newtheorem{prob}[thm]{Problem}
\newtheorem{rem}[thm]{Remark}

\theoremstyle{remark}

\DeclareMathOperator{\id}{id}
\DeclareMathOperator{\isoto}{\overset{\scriptstyle{\sim}}{\to}}

\DeclareMathOperator{\Ker}{Ker}
\DeclareMathOperator{\im}{Im}

\DeclareMathOperator{\kos}{Kos}

\DeclareMathOperator{\Tordim}{Td}

\newcommand{\TOR}{\operatorname{\mathcal{TOR}}}
\newcommand{\Tor}{\operatorname{Tor}}
\DeclareMathOperator{\Sym}{Sym}
\DeclareMathOperator{\Gr}{Gr}

\DeclareMathOperator{\Codim}{Codim}

\DeclareMathOperator{\hight}{ht}
\DeclareMathOperator{\Spec}{Spec}
\DeclareMathOperator{\Supp}{Supp}
\DeclareMathOperator{\Supph}{Supph}

\DeclareMathOperator{\op}{op}

\DeclareMathOperator{\a1}{\mathbb{A}^1}
\DeclareMathOperator{\p1}{\mathbb{P}^1}
\newcommand{\gm}{{\mathbf{G}}_m}
\newcommand{\ga}{{\mathbf{G}}_a}

\DeclareMathOperator{\Hom}{Hom}
\DeclareMathOperator{\HOM}{\mathcal{HOM}}

\DeclareMathOperator{\Homo}{H}
\DeclareMathOperator{\M}{M}

\DeclareMathOperator{\mm}{\mathfrak{m}}

\DeclareMathOperator{\AAA}{\mathcal{A}}
\DeclareMathOperator{\BBB}{\mathcal{B}}
\DeclareMathOperator{\CCC}{\mathcal{C}}
\DeclareMathOperator{\DDD}{\mathcal{D}}
\DeclareMathOperator{\EEE}{\mathcal{E}}
\DeclareMathOperator{\FFF}{\mathcal{F}}
\DeclareMathOperator{\GGG}{\mathcal{G}}
\DeclareMathOperator{\HHH}{\mathcal{H}}
\DeclareMathOperator{\III}{\mathcal{I}}

\DeclareMathOperator{\LLL}{\mathcal{L}}
\DeclareMathOperator{\MMM}{\mathcal{M}}
\DeclareMathOperator{\NNN}{\mathcal{N}}
\DeclareMathOperator{\OOO}{\mathcal{O}}

\DeclareMathOperator{\SSS}{\mathcal{S}}
\DeclareMathOperator{\TTT}{\mathcal{T}}

\DeclareMathOperator{\XXX}{\mathcal{X}}
\DeclareMathOperator{\YYY}{\mathcal{Y}}

\DeclareMathOperator{\QQQQ}{\mathbb{Q}}

\DeclareMathOperator{\Ab}{\mathbf{Ab}}

\newcommand{\Perf}{\operatorname{\bf Perf}}
\newcommand{\sPerf}{\operatorname{\bf sPerf}}

\DeclareMathOperator{\qis}{qis}
\DeclareMathOperator{\on}{\ \mathrm{on}\ }
\DeclareMathOperator{\naive}{naive}
\DeclareMathOperator{\DE}{\mathbf{DE}}
\DeclareMathOperator{\Nis}{Nis}
\DeclareMathOperator{\Zar}{Zar}

\UseAllTwocells

\newcommand{\cf}{\textrm{cf.}\;}

\newcommand{\onto}[1]{\stackrel{#1}{\to}}

\newcommand{\cA}{\mathcal{A}}
\newcommand{\cC}{\mathcal{C}}
\renewcommand{\cD}{\mathcal{D}}
\newcommand{\cE}{\mathcal{E}}
\newcommand{\cF}{\mathcal{F}}
\newcommand{\cG}{\mathcal{G}}
\newcommand{\Ch}{\operatorname{\bf Ch}}
\newcommand{\cI}{\mathcal{I}}
\newcommand{\cK}{\mathcal{K}}
\renewcommand{\cL}{\mathcal{L}}
\newcommand{\cM}{\mathcal{M}}
\newcommand{\cO}{\mathcal{O}}
\newcommand{\cP}{\mathcal{P}}
\newcommand{\cT}{\mathcal{T}}
\newcommand{\DM}{\operatorname{DM}}
\newcommand{\SH}{\operatorname{\bf{SH}^{\a1}}}

\newcommand{\Mod}{\operatorname{\bf Mod}}
\newcommand{\ModX}{\Mod(X)}

\newcommand{\MZ}{\mathbb{MZ}}

\newcommand{\opcit}{\textit{op.\,cit.}}
\newcommand{\OX}{\cO_X}
\newcommand{\OXx}{\cO_{X,x}}

\newcommand{\Eb}{E^{\bullet}}
\newcommand{\Pb}{P^{\bullet}}
\newcommand{\Perfqc}{\Perf_{\mathrm{qc}}}
\newcommand{\PerfX}{\Perf(X)}
\newcommand{\PerfXonY}{\Perf(\XonY)}
\newcommand{\PerfqcXonY}{\Perfqc(\XonY)}
\newcommand{\Qcoh}{\operatorname{\bf Qcoh}}
\newcommand{\QcohX}{\Qcoh(X)}
\newcommand{\QcohXonY}{\Qcoh(\XonY)}
\newcommand{\qpsm}{\operatorname{\bf qpsm}}
\newcommand{\qpsmcor}{\operatorname{\bf qpsmcor}}

\newcommand{\ssm}{\smallsetminus}

\newcommand{\inj}{\hookrightarrow}

\newcommand{\Wt}{\operatorname{\bf Wt}}
\newcommand{\Wtr}{\Wt^r}

\newcommand{\WtrXonY}{\Wtr(\XonY)}
\newcommand{\XonY}{X\!\on\!Y}
\newcommand{\Z}{\mathbb{Z}}
\newcommand{\ztr}{\mathbb{Z}_{tr}}
\newcommand{\UU}{\mathfrak{U}}

\def\sn{\smallskip\noindent}
\def\mn{\medskip\noindent}

\title{Deforming motivic theories I:\\
Pure weight perfect Modules on divisorial schemes}
\author{Toshiro Hiranouchi\footnote{
Supported by the JSPS Fellowships for Young Scientists.} and 
Satoshi Mochizuki\footnote{This research is supported by JSPS core-to-core program 18005
}}
\date{}

\begin{document}

\maketitle

\begin{abstract}
In this paper, 
we introduce a notion of weight $r$ 
pseudo-coherent Modules 
associated to a regular closed immersion 
$i:Y \inj X$ 
of codimension $r$, 
and prove that 
there is 
a canonical derived Morita equivalence 
between 
the DG-category of perfect complexes 
on a divisorial scheme $X$
whose cohomological support are 
in $Y$
and the DG-category of bounded complexes 
of weight $r$ pseudo-coherent $\OX$-Modules 
supported on $Y$. 
The theorem implies that 
there is the canonical isomorphism 
between 
the Bass-Thomason-Trobaugh non-connected $K$-theory 
\cite{TT90}, \cite{Sch06} 
(resp. the Keller-Weibel cyclic homology \cite{Kel98}, \cite{Wei96}) 
for the immersion and 
the Schlichting non-connected $K$-theory \cite{Sch04} associated to 
(resp. that of) 
the exact category 
of weight $r$ pseudo-coherent Modules. 
For the connected $K$-theory case, 
this result is just Exercise 5.7 in \cite{TT90}. 
As its application, 
we will decide on a generator 
of the topological filtration 
on the non-connected $K$-theory 
(resp. cyclic homology theory) 
for affine Cohen-Macaulay schemes.
\end{abstract}

\section{Introduction}

Since the word \lq\lq motive theory\rq\rq\ is 
an ambiguous word, 
in this Introduction, 
as motive theory, 
we restrictedly mean 
axiomatic studying (co)homology theories 
over algebraic varieties 
by enriching morphisms between 
algebraic varieties 
with adequate equivalence relations. 
Traditionally, 
to construct motivic categories, 
we used to choose 
certain classes of algebraic cycles 
as morphisms spaces 
and 
consider various equivalence relations on them, 
for example 
rational, numerical and algebraic relations and so on. 
In practice, 
the difficulty of handling a motivic theory 
is concentrating on moving algebraic cycles suitably 
in an appropriate equivalence relation class 
(see the proficient survey \cite{Lev06}). 
A problem of this type is so-called \lq\lq moving lemma\rq\rq\ 
and solving 
by deliberating on geometry over a base (field). 
In this paper, 
we give a first step of 
building up a motivic theory which does not rely upon geometry 
over a base 
by replacing (moduli spaces of) algebraic cycles 
with (roughly speaking, moduli non-commutative spaces of) 
pseudo-coherent complexes 
and considering an equivalence relation 
on them 
as the derived Morita equivalences 
(Compare \cite{Kon07} \S 4, \cite{Tau07}).

The aim of this paper is 
to introduce the notion of ({\it{Thomason-Trobaugh}}) {\it{weight}} 
on the class of perfect Modules 
on schemes inspired 
by the work of Thomason and Trobaugh in \cite{TT90}. 
To explain this more precisely, 
let $X$ be a divisorial scheme 
(in the sense of \cite{BGI71}, 
\cf Def.~\ref{divisorial df}) 
and $i:Y \inj X$ 
a regular closed immersion of codimension $r$. 
A pseudo-coherent $\OX$-Module 
is said to be {\it{of}} ({\it{Thomason-Trobaugh}}) {\it{weight}} $r$ 
supported on $Y$ 
if it is of Tor-dimension $\leqq r$ 
and supported on $Y$. 
Here the word \lq\lq weight\rq\rq\ 
is coming from 
the weight of the Adams operations in \cite{GS87} 
and a more systematic study will be done in \cite{Moc08}. 
We denote by $\WtrXonY$ 
the exact category of pseudo-coherent $\OX$-Modules 
of weight $r$ supported 
on the subspace $Y$ 
and $\PerfXonY$ 
the exact category of perfect complexes 
on $X$ whose cohomological support 
are in $Y$. 
We shall prove the following theorem:

\begin{main}[Th.~\ref{main theorem}]
There is a canonical derived Morita equivalence between 
the exact category of bounded complexes of $\WtrXonY$ and 
$\PerfXonY$.
\end{main} 

As alluded to above, 
it can be considered 
as one of a variant 
of \lq\lq moving lemma\rq\rq. 
It might sound a new flavored theory, 
but the methods of proving Theorem~\ref{main theorem} 
are classical, standard 
and almost all of them 
were established 
by Grothendieck school. 
For example, 
Verdier's coherator theory (Prop.~\ref{coherator prop}), 
Ilusie's global resolution theorem (Th.~\ref{global resol thm}), 
Grothendieck's local cohomology theory (Lem.~\ref{local cohomology lem})
and so on. 
The theorem implies that 
there is a canonical isomorphism 
between the Bass-Thomason-Trobaugh non-connected $K$-theory $K^B(\XonY)$ 
\cite{TT90}, \cite{Sch06} 
(resp. the Keller-Weibel cyclic homology 
$HC(\XonY)$ \cite{Kel98}, \cite{Wei96}) 
and 
the Schlichting non-connected $K$-theory \cite{Sch04} associated to 
(resp. that of) the exact category 
of weight $r$ pseudo-coherent $\OX$-Modules 
$K^S(\Wt(\XonY))$
(resp. $HC(\Wt(\XonY))$). 
That is, we have isomorphisms 
\begin{gather*}
  K_q^B(\XonY) \simeq K^S_q(\WtrXonY), \\
  HC_q(\XonY) \simeq HC_q(\WtrXonY),
\end{gather*}
for each $q\in \Z$.
For the connected $K$-theory
this result is nothing other than Exercise~5.7 in \cite{TT90}. 
For Grothendieck groups ($q=0$),
there is a detailed proof 
if $X$ is the spectrum of a Cohen-Macaulay local ring 
and $Y$ is the closed point of $X$ 
(\cite{RS03}, Prop.~2). 
For $K$-theory, 
as mentioned in Exercise 5.7, 
this problem is related with the works
\cite{Ger74}, \cite{Gra76} and \cite{Lev88}.
Namely the problem about 
describing the homotopy fiber 
of $K^B(X) \to K^B(X \ssm Y)$ 
(or rather than $K^Q(X) \to K^Q(X \ssm Y)$) 
by using the $K$-theory of a certain exact category. 
As described in \cite{Ger74}, 
there is an example due to Deligne which 
suggests difficulty of the problem 
for a general closed immersion. 
Conversely,
the example indicate that for an appropriate scheme $X$, 
there is a good class of pseudo-coherent $\OX$-Modules. 
That is, Modules of pure weight. 
This concept is intimately related 
to Weibel's $K$-dimensional conjecture \cite{Wei80} 
(see Conj.~\ref{Weibel conj}), 
Gersten's conjecture \cite{Ger73} 
and its consequences. 
These subjects will be treated 
in \cite{HM08}, \cite{Moc08}.
Notice that 
there are different notions of pure weight 
by Grayson \cite{Gra95} 
and Walker \cite{Wal00} and
these two notions are compatible 
in a particular situation \cite{Wal96}. 
In a future work, 
the authors hope 
to compare the Walker weight 
with the Thomason-Trobaugh one 
by utilizing 
the ({\it{equidimensional}\/}) {\it{bivariant algebraic $K$-theory}} 
\cite{GW00}. 

Now we explain the structure of the paper. 
In \S 2, we describe to our motivational picture. 
After reviewing the fundamental facts in \S 3, 
we will define the notion of weight 
and state the main theorem in \S 4. 
The proof of the main theorem  will be given in \S 5. 
Finally we will give applications of the main theorem in \S 6.

\mn
{\it Convention.} 
Throughout this paper, 
we use the letter $X$ to denote a scheme. 
A {\it complex}\/ means 
a chain complex whose boundary morphism is increase level 
of term by one. 
For fundamental notations of chain complexes, 
for example mapping cone and mapping cylinder etc..., 
we follow the book \cite{Wei94}. 
For an additive category $\cA$, 
we denote by 
$\Ch(\cA)$ the category of chain complexes in $\cA$. 
The word ``$\OX$-Module'' means 
a sheaf on $X$ which is a sheaf of modules 
over the sheaf of rings $\OX$. 
We denote by 
$\Mod(X)$ the abelian category of $\OX$-Modules and 
$\QcohX$ the category of quasi-coherent $\OX$-Modules. 
An {\it algebraic vector bundle} over the scheme $X$ 
is a locally free $\OX$-Module of finite rank and 
we denote by 
$\cP(X)$ the category of algebraic vector bundles. 
In particular a {\it line bundle} 
is an algebraic vector bundle of rank one (= an invertible sheaf). 
For the terminologies of algebraic $K$-theory, 
we follow to the notations in \cite{Sch07}. 
For example, 
for a complicial biWaldhausen category $\cC$, 
we denote its associated derived category by 
$\cT(\cC)$ and for an exact category $\cE$, 
we denote its associated derived category $\cT(\Ch(\cE))$ by $\cD(\cE)$. 
Finally for the $\a1$-motivic theory, we follow the notations in \cite{MVW06}. 

\mn
{\it Acknowledgments.} 
The second author is thankful to Masana Harada, 
Charles A. Weibel for giving several comments to Exercise 5.7 in \cite{TT90}, 
Marco Schlichting for teaching about elementary questions of 
negative $K$-theory 
via e-mail, 
Paul Balmer for bringing him to the preprint \cite{Bal07} 
and Mark E. Walker for sending the thesis \cite{Wal96} to him.

\section{Conjectural picture}

In this section, 
we will give a conjectural perspective 
of {\it{deforming motivic theories}}. 
This section is logically independent of the others.

\subsection{Analogies between multiplicative and additive motivic theories}
{\label{Analogies}}

As in the Introduction, 
as a motive theory, 
we prefer to mean 
axiomatic studying of 
(co)homology theories 
over algebraically geometric objects
by enriching morphisms 
between 
algebraically geometric objects 
with adequate equivalence relations. 
So there should be many motivic theories depending 
on our treating of algebraically geometric objects and (co)homology theories. 
For example, if we deal with Weil cohomology theories, 
the classical motive theory is fitting 
for our purpose \cite{Kle68}. 
If we handle $\a1$-homotopy invariant (co)homology theories, 
the motivic homotopy theory 
in the sense of Voevodsky is appropriate \cite{Voe00}. 
If we consider cohomology theories 
which has the Gersten resolution, 
the Bloch-Ogus(-Gabber) theory 
\cite{BO74}, \cite{CHK97} is suitable. 
Moreover there are other motivic theories 
for example \cite{KS02}, \cite{KL07}. 
It might be believed that 
there is \lq\lq the\rq\rq\ motive theory 
which is omniscient 
and unifying every motivic theories. 
But as in the following example, 
there are motivic theories 
which are not seemed 
to be compatible with each other.

\begin{ex}
If we prefer to give 
a motivic interpretation 
of the Hodge decomposition 
using the cyclic homology theory 
like as \cite{Wei97}, 
or 
if we like to understand 
what is the motive associated 
with the additive group $\ga$ 
like as a generalized $1$-motive \cite{Lau96}, \cite{Ber08}, 
we shall not realize them 
in Voevodsky's motivic world. 
For the cyclic homology theory 
is not $\a1$-homotopy invariant and 
$\ga$ is contractible 
in his motivic category. 
But we have the analogies table 
between additive 
and multiplicative worlds 
as in \cite{Lod03}. 
We would like to extend the table to motivic stage. 
For example, 
the Bloch theorem \cite{Blo86} 
and the Hodge decomposition as in the table below.\\
\\
\begin{tabular}{c|c}
\hline
$\times$ (Voevodsky's motivic theory) & $+$ (additive motivic theory)\\
\hline
$K_n(X)_{\QQQQ}\isoto \underset{p+q=n}{\bigoplus} \Homo_{\MMM}^p(X,\gm^{\otimes q})_{\QQQQ}$ &
$\Homo^n(X,\mathbb{C}) \isoto \underset{p+q=n}{\bigoplus} \Homo^p(X,\ga\otimes \gm^{\otimes q})$ \\
\hline
\end{tabular}
\\
\\
Of course, 
the right hand side above 
is conjectural description. 
(But see \cite{BE03}, \cite{Rul07}, \cite{Par07} and \cite{Par08}). 
In these analogical line, 
following \cite{FT85} and \cite{FT87}, 
we like to call the cyclic homology theory 
the {\it{additive algebraic $K$-theory}}.
\end{ex}

We shall also notice the fact 
that there are real mathematical problems 
stretching away
both additive and multiplicative worlds. 
For example, 
Vorst's conjecture \cite{Vor79}. 
Actually the conjecture is proved 
in a special case 
by frequently utilizing 
both multiplicative 
and additive motivic techniques \cite{CHW06}. 
In the next subsection 
we propose another similarly kind problems.

\subsection{Motivic modules and Weil reciprocity law}
{\label{Motivic modules}}

Classically there is the following problem.

\begin{prob}
{\label{somekawa problem}}
Let $G_1,\ldots,G_r$ be commutative group varieties 
over a base field $k$. 
Then we have the correspondence
$$
 Z_0(G_1 \times_k G_2 \times_k \cdots \times_k G_r) 
\leftrightarrow
\underset{L/k:\substack{\text{finite} \\ \text{extension}}}
{\bigoplus} 
G_1(L) \otimes_\Z G_2(L) \otimes_\Z \cdots \otimes_\Z G_r(L)
$$
where $Z_0(?)$ means 
the group of zero cycles. 
The problem is the following: \\
{\it{What are the suitable equivalence relations 
making assignment above isomorphism.}}
$$
 Z_0(G_1 \times_k G_2 \times_k \cdots \times_k G_r) /\sim
\isoto
\underset{L/k:\substack{\text{finite} \\ \text{extension}}} 
{\bigoplus}
G_1(L) \otimes_\Z G_2(L) \otimes_\Z \cdots \otimes_\Z G_r(L)/\sim 
.$$
\end{prob}

\begin{his}
{\label{somekawa his note}}
If we assume all $G_1,\ldots,G_r$ are semi-abelian varieties, 
then there are suitable candidates for
equivalence relations above. 

\sn
(i) 
In the left hand side, 
the suitable equivalence relation should 
come from the tensor products
as $1$-motives 
in the sense of \cite{Del74}.
That is, the left hand side should be 
replaced with
$$
  \Gamma(\Spec k, G_1\otimes \cdots \otimes G_r) 
= Z_0(G_1 \times_k G_2 \times_k \cdots \times_k G_r)/ \sim
$$
where tensor product are taken as $1$-motives. 

\sn
(ii) 
In the right hand side, 
Kazuya Kato proposed that 
the suitable equivalence relation should
be the following two relations. 

\sn
$\bullet$ Projection formula for norms.\\
$\bullet$ Weil reciprocity law for semi-abelian varieties.\\

We will write the left hand side modulo equivalence relations above as

$$
 K(k,G_1,\ldots,G_r)
$$
and called it 
{\it{Milnor $K$-group associated with $G_1,\ldots,G_r$}} 
(see for example \cite{Som90}, \cite{Kah92}).
The naming coming from the following isomorphism.
$$
 K(k,\overbrace{\gm,\ldots,\gm}^{r}) \isoto K^M_r(k).
$$
\end{his}

\begin{ob}
(i) 
(At least after tensoring with $\QQQQ$,) 
the tensor product as $1$-motives is equal to the tensor 
product in the $\a1$-motivic category $\DM(k)$ 
(see for example \cite{Org04}, \cite{BK07}). 

\sn
(ii) 
The projection formula relation above is 
one of the consequence of presheaf with transfer, 
that is, 
there is the following statement (see for example \cite{Org04}): 

\sn
{\it{Every commutative group variety 
over a field $k$ 
is considered as a functor
$$
 \qpsmcor(k) \to \Ab
$$
where $\qpsmcor(k)$ is the category 
of quasi-projective smooth varieties 
whose morphisms are 
finite surjective correspondences.}} 

\sn
(iii) 
If all $G_i$ are $1$-dimensional semi-abelian varieties, 
we have the following formula,
$$
 \Gamma(\Spec k,G_1\otimes,\cdots \otimes G_r) \isoto 
K(k,G_1,\ldots,G_r).
$$
This is the affirmative answer for question above 
(see for $\gm$ case \cite{SV00} and for elliptic curve case \cite{Moc06}).

\sn
(iv) The reason why semi-abelian varieties 
are fit in Voevodsky's theory is 
that semi-abelian varieties 
are $\a1$-homotopy invariant presheaves 
with transfers. 
So we shall say $\a1$-homotopy 
invariant presheaves with transfers 
as a {\it{motivic modules}}. 
Then we can re-write the statement
in Historical Notes~\ref{somekawa his note} (i) as follows. 

\sn
{\it{In the left hand side, 
the suitable equivalence relation should 
come from the tensor products as
motivic modules.}}
\end{ob}

In the observation and \S~\ref{Analogies}, 
we are interested in the following question. 

\sn
{\it{What is a good notion of motivic modules including $\ga$?}}

\begin{rem}
(i) (\cf \cite{RO06}) The category of motives 
is reinterpreted in the context of stable motivic homotopy theory
by R\"ondings and {\O}stv{\ae}r as follows. 
Let $\MZ \in \SH (k)$ be the motivic Eilenberg-Maclane
spectrum. Then $\MZ$ is considered as 
a ring object in $\SH (k)$ in the natural way 
and we have the following identity:
$$
 \Mod(\MZ) \isoto
 \DM(k)
$$
where we assume that characteristic of $k$ is zero. 
This means $\DM(k)$ is actually 
\lq\lq the category of motivic modules\rq\rq\ in some sense. 
Notice that 
if a presheaf of abelian groups 
on the category of quasi-projective smooth schemes 
has an action of $\MZ$, 
this means that $F$ can extend to a presheaf on $\qpsmcor(k)$. 

\sn
(ii) 
Several authors are attempting to describe 
$\Gamma(\Spec k, G_1 \otimes \cdots \otimes G_r)$ 
as generators and relations. 
In this point, 
relations are related with 
the functional equations of special functions
associated with $G_i$. 
For example, 
if all $G_i$ are equal to $\gm$, 
the special function is the polylogarithms \cite{Gon94} 
and so on. 
Therefore it is quite surprised 
that the relations 
of $K(k,G_1,\ldots,G_r)$ does not depend
on the $G_i$. 
The Weil reciprocity law is 
implicitly controlling the functional equations of 
special functions associated with $G_i$. 
So it is important that we shall ask what is a meaning of 
the Weil reciprocity law 
in the context of Voevodsky's motivic theory.
\end{rem}

We can state a generalization of 
the Weil reciprocity law 
which is called {\it{Motivic reciprocity law}}. 
Let $k$ be a field 
which satisfies the resolution of singularity assumption.

\begin{thm}[\cite{Moc06}]
For a field extension 
of transcendental degree one $K/k$, 
the composition of
\begin{equation}
{\label{MRL}}
\M(\Spec k)(1)[1]
\onto{\Sigma N_{k(v)/k}(1)[1]}
\underset{v:\substack{\text{place} \\ \text{of } K/k}}
{\tilde{\prod}}
\M(\Spec k(v))(1)[1]
\onto{\tilde{\prod} \partial_v}
\M(\Spec K) 
\end{equation}
is the zero map in the pro-category of $DM(k)$.
\end{thm}

If we take 
the $\Hom_{\DM(k)}(?,\Z(n + 1)[n + 1])$ 
for the sequence (\ref{MRL}), 
we can easily reprove the Weil reciprocity law for Milnor K-groups.

\begin{cor}[\cite{Sus82}]
The composition of
$$
 K^M_{n+1}(K) \onto{\oplus \partial_v} 
 \underset{v:\substack{\text{place} \\ \text{of } K/k}}{\bigoplus}
 K^M_n(k(v)) \onto{\Sigma N_{k(v)/k}}  
 K^M_n(k)
$$
is the zero map.
\end{cor}

The crucial point of proving the motivic reciprocity law is 
the existence of functorial Gysin triangles which is 
proved by D\'eglise \cite{Deg06} 
and $\a1$-homotopy invariance is indispensable 
in his construction of the triangle. 
On the other hand, 
R{\"u}lling proved the Weil reciprocity law 
for the de Rham-Witt complexes 
which is not an $\a1$-homotopy invariant theory \cite{Rul07}. 
We would like to explain this 
reciprocity law also in the context 
of an alien motivic theory. 
In this way, 
we sometimes have interested in 
the problems 
which stretching away several motivic theories and 
sometimes intend to analyze relationship 
of several motivic theories, 
for example, 
their analogies and differences. 
The main theme of {\it{deforming motivic theories}} is 
investigating the relationship 
between various motivic theories. 
In particular, 
Voevodsky's motivic theory 
and an alien (additive) motivic theory.

\subsection{How to describe deforming motivic theories I}

Next we intend to 
illustrate 
how to describe 
deforming motivic theories. 
As in \cite{Han95}, \cite{RO06} and \cite{BV07}, 
the triangulated category of motivic sheaves shall be 
the connected components 
of the $\infty$-category 
of that in some sense. 
Here the word \lq\lq $\infty$-category\rq\rq\ means 
(quasi-) DG-category 
or $\SSS$-category 
in the sense of T\"oen and Vezzosi \cite{TV04}. 
We first start to consider
how to mention 
an alien motivic theory 
as follows.

\begin{ex}
(i) (Toy model)  
Let $V$ be a finite dimensional vector space 
over a field $k$ 
with an inner product 
and $W$ its sub vector space. 
Then we have an isomorphism
\begin{equation}
{\label{topos equality}}
V/W \isoto W^{\bot}
\end{equation}
where $W^{\bot}$ is the orthogonal subspace of $W$ in $V$.

\sn
(ii) 
Let $k$ be a perfect field, 
$V$ the derived category of complexes 
of Nisnevich sheaves transfer over $k$
bounded from above 
and $W$ the localizing subcategory generated 
by the complexes of the form
$$\ztr(X\times \a1) \to \ztr(X)$$
for smooth schemes $X$ over $k$. 
Then we have the equivalence (\ref{topos equality}) 
where $W^{\bot}$ is the full subcategory 
of those complexes 
whose cohomology sheaves are $\a1$-homotopy invariant 
in $V$ (see \cite{Voe00}, Prop.~3.2.3). 
The sign $\bot$ is justified 
in the context of (generalized) topoi theory 
or Bousfiled localization theory as below.

\sn
(iii) (\cf \cite{BGV72}, IV) 
Let $\CCC$ be a small category 
with a Grothendieck topology $\tau$. 
We denote the category 
of presheaves on $\CCC$ 
by $V$ and 
the category 
of $\tau$-local contractible presheaves 
on $\CCC$ by $W$. 
Then we have an equivalence (\ref{topos equality})
where $W^{\bot}$ is the full subcategory 
of $\tau$-sheaves in $V$. 
Namely, an object $F$ in $W^{\bot}$ 
is satisfying the decent condition 
(or rather than saying the orthogonal condition) 
as follows:
$$
  \Hom(\UU,F)\isoto\Hom(h_X,F)
$$ 
where $h_X$ is 
the functor represented 
by an object $X$ in $\CCC$ 
and $\UU$ is a crible in $\tau(X)$. 
As in \cite{Hir03}, \cite{TV05}, 
replacing 
$\CCC$ as above 
with a more higher categorical 
(or rather than say homotopical) 
object in some sense, 
the argument above still works fine 
by replacing the decent condition 
with the hyper one 
(For precise statement, 
consult with \cite{TV05}). 
For DG-categories case, 
see \cite{Dri04} and \cite{Tau07} Appendix.
\end{ex}

Now we would better consider 
the reason why 
a hyper descent condition 
is not seemed to be involved 
in Voevosky's $\a1$-homotopy theory. 
To do so, 
let us recall the following Lemma~\ref{characterization of homotopy inv}: 

\begin{df}
Let $(I, x, y)$ be a triple consisting of 
$I\in \qpsmcor(k)$ 
and different $k$-rational points $x:\Spec k \to I$, 
$y:\Spec k \to I$. 

\sn
(i) Two maps $f$, $g:X\to Y$ in $\qpsmcor(k)$ are said to be 
{\it{$I$-homotopic}} if there is a map 
$H:X\times I \to Y$
such that 
$H \circ x \times \id_X = f$ and $H \circ y \times \id_Y = g$ 
(or $H \circ x \times \id_X = g$ and $H \circ y \times \id_Y = f$).

\sn
(ii) 
A functor $F:\qpsmcor(k) \to \CCC$ is said to 
be {\it{$I$-homotopy invariant}} 
if for any $I$-homotopic maps 
$f$, $g:X \to Y$, $F(f) = F(g)$. 
\end{df}

For $\a1$-homotopy invariant, 
we mean 
that $\a1$-homotopy invariant 
for the triple $(\a1,0,1)$.

\begin{lem}
{\label{characterization of homotopy inv}}
For any presheaf $F$ on $\qpsmcor(k)$, 
the following conditions are equivalent. 

\sn
$\mathrm{(i)}$ 
For any scheme $X$, 
the projection $X\times \a1 \to X$ 
induces an isomorphism 
$$F(X\times\a1)\isoto F(X).$$ 

\sn
$\mathrm{(ii)}$
$F$ is $\a1$-homotopy invariant. 
\end{lem}

Notice that 
the condition $\mathrm{(i)}$ (resp. $\mathrm{(ii)}$) above 
is a descent condition for objects (=$0$-morphisms) 
(resp. morphisms (=$1$-morphisms)) 
in some sense. 
Therefore the condition $\mathrm{(i)}$ 
is seemed to be stronger than 
the condition $\mathrm{(ii)}$. 
We designate that
to prove $\mathrm{(i)}$ 
from the condition $\mathrm{(ii)}$, 
we are using the special feature of $\a1$. 
Namely the existence of the multiplication 
$\a1 \times \a1 \ni (x,y) \mapsto xy \in \a1$ 
and this feature is axiomized 
by Voevodsky 
as the site with interval theory \cite{Voe96}, \cite{MV99}. 
In the authors view point, 
this is the reason why 
we are able to shortcut 
to construct the motivic homotopy category 
without using a hyper descent theory 
and there is no reason 
that to establish an alien motivic category, 
we can avoid using 
a higher topoi theory. 
So we propose the following. 

\begin{con}[Very obscure version]
{\label{Very obscure version}}
To build up an alien motivic category, 
we need to choose a moduli space $V$ 
of algebraically geometric objects 
which could be represented 
by $\infty$-category 
or homotopical category 
as a generator class 
and a moduli space $W$ 
of relations space. 
Then we can define 
an alien motivic category 
by the quotient space $V/W$ 
and somewhat hyper descent theory 
implies that 
it is equivalent to $W^{\bot}$. 
Here $W^{\bot}$ is full subspace of $V$ consisting 
of the objects which satisfy 
orthogonal condition in some sense. 
\end{con}

Obviously the conjecture has two faces. 
One face is the problem of 
establishing the general frame works 
of a higher or generalized topoi theory 
fitting for our purpose. 
For example, 
presheaves with transfer theory 
and site with interval theory 
can be considered 
as a sheave theory over generalized Grothendieck topology 
and are suitable 
for describing $\a1$-motivic theory. 
The other face is the problem of 
finding the good class of $V$ and $W$ above. 
To attack the first face, 
we need drastically axiomatic consideration. 
To study the second one,
we need look squarely at real many examples. 
The authors are starting from attacking to the second one. 
After getting many important examples, 
they intend to contemplate the first one \cite{HM08}. 

\subsection{Thomason categories and bivariant algebraic $K$-theory}

It is a complicial biWaldhausen category 
closed under the formation 
of the canonical homotopy push-outs 
and pull-backs 
in the sense of \cite{TT90} 
that makes sense of 
its derived category 
and algebraic (resp. additive) $K$-theory 
and we like to call it a {\it{Thomason category}}. 
A morphism between 
Thomason categories 
is complicial exact functor 
in the sense of \opcit\  
In this paper, we examine 
$V$ in Conjecture~\ref{Very obscure version} as
the category of 
Thomason categories 
which is a homotopical category 
in the sense of \cite{DHKS04} 
by declaring the class of weak equivalences 
as {\it{derived Morita equivalences}}, 
that is, morphisms 
which induce equivalences 
of derived categories. 
The reasons why we prefer to 
take algebraically geometric objects 
as Thomason categories 
are the following: \\
\\
$\bullet$ There is a functor 
from the category of schemes 
to that of Thomason categories: 
$X \mapsto \PerfX$, where $\PerfX$ is 
the category of 
perfect complexes of globally finite Tor-amplitude (\cf \cite{TT90}, \S 2.2). 

\sn
$\bullet$ For an appropriate scheme $X$, 
from $\PerfX$ (and its tensor structure), 
we can recover the scheme $X$ completely 
\cite{Bal02}. 
That is, 
$\Perf(X)$ does not lose the geometric information of $X$. 

\sn
$\bullet$ Moreover in the category of Thomason categories, 
we have objects like $\PerfXonY$ and $\Perf^r(X)$ (\cf Def.~\ref{perf df}) 
which are derived from schemes and 
absent from 
the category of schemes. 

\sn
$\bullet$ Since the algebraic $K$-theory is
$\infty$-categorical invariant 
(see \cite{Sch02}, \cite{Toe03}, \cite{TV04} and \cite{BM07}),
we prefer to the category 
of Thomason categories 
than that 
of triangulated categories. \\

Next we need to consider 
how to enrich the category $V$ 
and choose a relation space $W$. 
Inspired from the work 
\cite{Wal96} and 
encouraged by the works 
\cite{Kon07} \S 4 and \cite{Tau07}, 
the authors intend to enriching $V$ 
with the bivariant algebraic $K$-theory. 
To mention the reason why 
we like to select the bivariant $K$-theory 
as morphisms spaces of $V$, 
we will start from 
the following Lemma~{\ref{characterization of p1-homotopy invariance}}. 
Let $\DDD$ be a tensor triangulated category 
and $\M:\qpsmcor(K) \to \DDD$ 
a functor preserving coproducts and tensor products. 
Here the tensor products 
in $\qpsmcor(k)$ are the usual products over $\Spec k$. 
From now on, 
for $\p1$-homotopy invariant, 
we mean that $\p1$-homotopy invariant 
for the triple $(\p1,0,1)$.

\begin{lem}[Compare \cite{CHK97}]
{\label{characterization of p1-homotopy invariance}}
The following
conditions are equivalent. 

\sn
$\mathrm{(a)}$ 
$\M$ is $\p1$-homotopy invariant. 

\sn
$\mathrm{(b)}$ 
The following diagram is commutative.
$$\xymatrix{
\M(\a1 \ssm \{0\}) \ar[rr]^{\M(i)} \ar[rd]_{\M(p)} & & \M(\p1)\\
 & \M(\Spec k) \ar[ru]_{\M(\infty)}
}$$
where $i$ and $p$ are the natural inclusion 
and the structure map respectively. 

\sn
$\mathrm{(c)}$ (Rigidity) $\M(1) = \M(\infty) : \M(\Spec k) \to \M(\p1)$ 
is coincided. 
\end{lem}

\begin{proof}
(a) $\Rightarrow$ (b): 
The map
$$H : (\a1 \ssm \{0\}) \times \p1\ni  
(t, [x_0 : x_1]) \mapsto [tx_0 : tx_1] \in \p1$$
gives $\p1$-homotopy between $i$ and $\infty \circ p$. 
Therefore we get the results. 

\sn
(b) $\Rightarrow$ (c): 
Considering the following diagram, 
we get the result. 
$$\xymatrix{
\M(\Spec k) \ar[r]^{\M(0)} \ar[d]^{\id} & \M(\a1 \ssm \{0\}) \ar[d]^{\M(i)} \ar[dl]^{\M(p)}\\
\M(\Spec k) \ar[r]_{\M(\infty)} & \M(\p1).
}$$

\sn
(c) $\Rightarrow$ (a): 
Let $f$, $g : X \to Y$ 
be maps in $\qpsmcor(k)$ and 
$H : X \times \p1 \to Y$ are their $\p1$-homotopy,
that is, 
$H \circ 0 \times \id_X = f$ and 
$H \circ 1 \times \id_X = g$. 
Then we have the identity:
\begin{align*}
\M(f) &= \M(H \circ 0 \times \id_X) 
= \M(H \circ \tau  \circ \infty\times \id_X)\\
 &= \M(H) \M(\tau ) \M(\infty) \otimes \M(\id_X)
= \M(H) M(\tau ) \M(1) \otimes \M(\id_X) \\
&= \M(H \circ \tau \circ  1 \times \id_X) = \M(H \circ 1 \times \id_X) = \M(g)
\end{align*}
where $\tau : \p1 \ni [x_0 : x_1] \mapsto [x_1 : x_0] \in \p1$.
\end{proof}

For the importance of the commutative diagram 
in Lemma~\ref{characterization of p1-homotopy invariance} (b), 
the readers shall consult with \cite{CHK97} 
and this topic will be treated in \cite{HM08}. 
It is closely related to the existence 
of the Gersten resolution for $\M$. 
We also notice that the additive $K$-theory, 
additive higher Chow groups 
and the additive group $\ga$ are $\p1$-homotopy invariant 
as functors on the category of algebraic varieties 
(see for example \cite{Qui73}, \cite{TT90}, \cite{Kel99}, \cite{KL07}). 
But $K_0$ is not a functor on $\qpsmcor(k)$. 
As in \S~\ref{Motivic modules}, 
we sometime hope to extend the notion 
of motivic modules
to make functors above belong to the class 
of generalized motivic modules. 
Imitating Walker's argument, 
we prefer to replace $\qpsmcor(k)$ with $K_0^{\naive}(\qpsm(k))$ 
which is the category 
of quasi-projective smooth schemes 
over $k$ enriching with the bivariant $K$-theory 
(For precise definition, see \cite{Wal96}, \cite{Sus03}). 
We like to call $\p1$-homotopy invariant presheaves 
of abelian groups on $K_0^{\naive}(\qpsm(k))$ 
{\it{generalized motivic modules}}. 
Now it is important that 
we recall the following core theorem 
of the $\a1$-motivic theory. 
Let us assume that $k$ is a perfect field.

\begin{thm}
{\label{Voevodsky's theorems}}
For an $\a1$-homotopy invariant presheaf with transfer $F$, 
we have the following. 

\sn
$\mathrm{(i)}$ 
For any $p$, 
$\Homo^p_{\Nis}(?, F_{\Nis})$ can be considered as an 
$\a1$-homotopy invariant Nisnevich sheaf with transfer in the natural way. 
That is, Nisnevich motivic modules are closed under taking cohomology. 

\sn
$\mathrm{(ii)}$ 
$\Homo^p_{\Nis}(?, F_{\Nis})\isoto 
\Homo^p_{\Zar}(?, F_{\Zar})$ for any $p$.
\end{thm}

Now Beilinson and Vologodsky perceived that 
Theorem~\ref{Voevodsky's theorems} is a consequence 
of the existence of the Gersten resolution of $F$ \cite{BV07} and 
Walker proved that for $\a1$-homotopy invariant presheaves on 
$K_0(\qpsm(k))$, similar theorem above are verified \cite{Wal96}. 
Therefore the touchstone 
of a notion of generalized motivic modules are following.

\sn
$\bullet$ For a generalized motivic module, 
does it have the Gersten resolution ? 

\sn
$\bullet$ Does the (equidimensional) bivariant $K$-theory 
have the expected properties like as the Friedlander-Voevodsky theory \cite{FV00} ? \\

In this paper, 
the authors prepare to attack to the second problem above. 
More precisely saying, 
in this paper and \cite{Moc08}, 
the authors will observe that 
the roll of a base of our motivic theory, 
analyze how to avoid to the geometry over the base 
(see \S~\ref{How to II}). 
Symbolically, 
let us denote $\star$ the invisible base for our motivic theory. 
If there exist a bigraded bivariant $K$-theory for schemes, 
in particular we can consider $K_{p,q}(X,\star)$ for a scheme $X$. 
The second author believe that 
$K_{p,q}(X,\star)$ might be $K_{p}(\Wt^q(X))$ and 
for a regular noetherian affine scheme $X$, 
the isomorphism 
$$K_{p}(\Wt^q(X))\isoto K_{p}(\MMM^q(X))$$ 
could be considered as 
a variant of Friedlander-Voevodsky duality theorem \cite{FV00}. 

\subsection{How to describe deforming motivic theories II}
\label{How to II}

To compare with two motivic theories, 
the author intend to parametrize the relation space $W$ 
in Conjecture~\ref{Very obscure version}. 
Namely for example we consider moduli space 
of motivic theories $V/W(t)$. 

\begin{ex}
\label{Sus func}
Let $R$ be a commutative discrete valuation ring and $\pi$ its uniformizer. 
We put $K=R[1/\pi]$ and $k=R/\pi R$. 
Then we can consider the {\it{parametrized Susulin functor}} 
by using the following parametrized cosimplicial scheme $\Delta_{\bullet}$. 
We define a parametrized cosimplicial scheme $\Delta_{\bullet}$ by 
$$[n] \mapsto \Spec R[T_0,\ldots,T_n]/(\Sigma T_i -\pi).$$
Obviously $\Delta_{\bullet}|_{\Spec K}$ is usual one and 
$\Delta_{\bullet}|_{\Spec k}$ is appeared in \cite{BE03}. 
\end{ex}

The attempt in Example~\ref{Sus func} is 
just a naive construction of deformation space of 
motivic theories parametrized by $\Spec R$ 
whose fiber over $\Spec K$ is the $\a1$-motivic theory and 
over $\Spec k$ is an alien one. 
But we are confronted with the following serious problems. 

\sn
{\it{What is the motivic theory of total space?}} 

\sn
{\it{Why does the total space theory work fine?}}\\

To solve the second problem above, 
we need to assure that 
we can build up a motivic theory 
without relying upon the geometry over a base. 
Therefore our deforming motivic theories is starting from 
examining the Thomason-Trobaugh weight. 

\section{Preliminary}
\subsection{Tor-dimension}

We briefly review the definition 
and fundamental properties of {\it{Tor-dimension}} of Modules. 

\begin{df} 
\label{Tor-dim def} 
Let $\cL$ be an $\OX$-Module.\ 

\sn
(i) $\cL$ is {\it flat} if the functor $?\otimes_{\OX}\cL:\Mod(X) \to \Mod(X)$ 
defined by $\cM\mapsto \cM\otimes_{\OX}\cL$ is exact.

\sn
(ii) A {\it Tor-dimension} of $\cL$ is the minimal integer $n$ 
such that there is a resolution of $\cL$,
$$
  0\to\cF_n\to\cF_{n-1}\to\cdots\to\cF_0\to\cL\to 0,
$$
where all $\cF_i$ are flat. 
We write as $\Tordim(\cL)=n$.
\end{df}

Now we list some well-known facts on Tor-dimension. 
\begin{lem}[\cite{BGI71}, Exp.~I, 5.8.3, \cite{DG63}, 6.5.7.1] 
\label{fund fact of flat module}
Let $\cL$ be an $\OX$-Module. 

\sn
$\mathrm{(i)}$ 
If $\cL$ is a flat and finitely presented $\OX$-Module, 
then $\cL$ is an algebraic vector bundle. 

\sn 
$\mathrm{(ii)}$
The following conditions are equivalent. 

\noindent
\quad $\mathrm{(a)}$ $\Tordim(\cL)\leqq d$.\\
\quad $\mathrm{(b)}$ For any $\OX$-Module $\cK$ and any $n>d$, we have $\Tor^{\OX}_n(\cL,\cK)=0$.\\
\quad $\mathrm{(c)}$ For any $\OX$-Module $\cK$, we have $\Tor^{\OX}_{d+1}(\cL,\cK)=0$.\\
\quad $\mathrm{(d)}$ If there is an exact sequence
$$0\to\cK_d\to\cF_{d-1}\to\cF_{d-2}\to\cdots\to\cF_0\to\cL\to 0$$
where all $\cF_i$ are flat, then $\cK_d$ is also flat.

\sn
$\mathrm{(iii)}$  For any short exact sequence of $\OX$-Modules
$$
  0\to\cL\to\cL'\to\cL''\to 0,
$$
we have a formula $\Tordim(\cL')\leqq \max\{\Tordim(\cL),\ \Tordim(\cL'')\}$.

\sn
$\mathrm{(iv)}$
For any $x\in X$ and quasi-coherent $\OX$-Modules $\cL$, $\cK$, we have
$$
  \TOR^{\OX}_n(\cL,\cK)_x\isoto\Tor^{\OXx}_n(\cL_x,\cK_x). 
$$
As its consequence, we have the following formula.
$$
  \Tordim(\cL)\leqq\underset{x\in X}{\sup}\Tordim_{\OXx}(\cL_x).
$$
\end{lem}
%
%

We define a similar Tor-dimension for unbounded complexes.
\begin{df}[\cite{TT90}, Def.~2.2.11]
\label{Tor-amp def}
Let $E^{\bullet}$ be a complex of $\OX$-Modules.

\sn
(i) $E^{\bullet}$ has ({\it{globally}}) {\it{finite Tor-amplitude}} 
if there are integers $a\leqq b$ and for all $\OX$-Module $\FFF$, 
$\Homo^k(E^{\bullet}\otimes_{\OX}^{L} \cF)=0$ unless $a\leqq k \leqq b$. 
(In the situation, 
we say that $E^{\bullet}$ {\it{has Tor-amplitude contained in $[a,b]$}}).

\sn
(ii) $E^{\bullet}$ has {\it{locally finite Tor-amplitude}} 
if $X$ is covered by opens $U$ such that $E^{\bullet}|_{U}$ 
has finite Tor-amplitude.
\end{df}

\begin{rem} \label{tor-amp rem} 
(i) If the scheme $X$ is quasi-compact, 
then every locally finite Tor-amplitude complex $E^{\bullet}$ of $\OX$-Modules 
is globally finite Tor-amplitude.

\sn
(ii) For 
three vertexes of a distinguished triangle in the derived category of 
$\ModX$, 
if two of these three vertexes are globally finite Tor-amplitude 
then the third vertex is also. 
\end{rem}

\subsection{The coherator}
We briefly review the theory of 
\lq\lq coherator\rq\rq\ from \cite{BGI71}, II 
and \cite{TT90} Appendix B.
There are two abelian categories $\QcohX$ and $\Mod(X)$ 
and the canonical inclusion functor $\phi_X:\QcohX \inj \Mod(X)$ which is exact, 
closed under extensions, reflects exactness, 
preserves and reflects infinite direct sums. 
The problem is that in general $\phi_X$ does not preserve injective objects in $\QcohX$. 
But for coherent schemes, there is a good theory for $\QcohX$. 
We are starting from reviewing the definition of coherence of schemes.

\begin{df}[\cite{DGSV72}, VI]\label{coherent scheme df}
The scheme $X$ is said to be {\it{quasi-separated}} 
if the diagonal map $X \to X\times X$ is quasi-compact or 
equivalently if intersection of any pair of affine open sets in $X$ is quasi-compact.
It is said to be {\it{coherent}} 
if it is quasi-compact and quasi-separated.
\end{df}

\begin{prop}[\cite{BGI71}, II, 3.2; \cite{TT90}, Appendix B]
\label{coherator prop}
Let $X$ be a coherent scheme. 
Then we have the following:

\sn
$\mathrm{(i)}$ $\phi_X$ has the right adjoint functor $Q_X:\Mod(X) \to \QcohX$ 
which is said to be {\rm coherator} 
and the canonical adjunction map $\id \to Q_X\phi_X$ is an isomorphism. 
In particular $\QcohX$ has enough injective and closed under limit.

\sn
$\mathrm{(ii)}$ 
$Q_X$ preserves limit.

\sn
$\mathrm{(iii)}$ 
For any $E^{\bullet} \in \DDD(\QcohX)$ and $F^{\bullet} \in \DDD(\Mod(X))$, 
the canonical adjunction maps $E^{\bullet} \to RQ_X\phi_XE^{\bullet}$ 
and $\phi_XRQ_XF^{\bullet} \to F^{\bullet}$ are quasi-isomorphisms.
\end{prop}

\subsection{Perfect and pseudo-coherent complexes}
We review the notion of pseudo-coherent and perfect complexes. 
For a complex of $\OX$-Modules $E^{\bullet}$ on $X$, 
perfection and pseudo-coherence 
depend only on the quasi-isomorphism class of $E^{\bullet}$ 
and are local properties on $X$. 
So first we define the strict version of them and 
next we define them as being local properties.

\begin{df}[\cite{BGI71}, Exp.~I; \cite{TT90}, \S~2.2]
\label{Pseudo-coherent def}
Let $E^{\bullet}$ be a complex of $\OX$-Modules. 

\sn
(i) 
$E^{\bullet}$ is {\it{strictly perfect}} (resp. {\it{strictly pseudo-coherent}}) 
if it is a bounded complex (resp. bounded above complex) of algebraic vector bundles.

\sn
(ii) 
$E^{\bullet}$ is {\it{perfect}} (resp. {\it{$n$-pseudo-coherent}}) 
if it is locally quasi-isomorphic (resp. $n$-quasi-isomorphic) to 
strictly perfect complexes. 
More precisely,  
for any point $x\in X$, 
there is a neighborhood $U$ in $X$, 
a strictly perfect complex $F^{\bullet}$, 
and a quasi-isomorphism (resp. an $n$-quasi-isomorphism) $F^{\bullet} \isoto E^{\bullet}|_{U}$. 
$E^{\bullet}$ is said to be {\it{pseudo-coherent}} 
if it is $n$-pseudo-coherent for all integer $n$.
\end{df}

\begin{lem}[\cite{TT90}, \S~2.2] 
\label{fund fact for psuedo-coh and perf} 
Let $E^{\bullet}$ be a complex of $\OX$-Modules on $X$. 

\sn
$\mathrm{(i)}$ If $E^{\bullet}$ is strictly pseudo-coherent, 
then it is pseudo-coherent.

\sn
$\mathrm{(ii)}$ In general, a pseudo-coherent complex may not be 
locally quasi-isomorphic to a strictly pseudo-coherent complex. 
But if $E^{\bullet}$ is pseudo-coherent complex of quasi-coherent $\OX$-Modules, 
then $E^{\bullet}$ is locally quasi-isomorphic to 
a strictly pseudo-coherent complex. 

\sn
$\mathrm{(iii)}$ If $E^{\bullet}$ is a pseudo-coherent, 
then all cohomology sheaf $\Homo^i(E^{\bullet})$ is quasi-coherent. 
In particular, a pseudo-coherent $\OX$-Module is a quasi-coherent $\OX$-Module.
Moreover if we assume $X$ is quasi-compact and $E^{\bullet}$ is pseudo-coherent, 
then $E^{\bullet}$ is cohomologically bounded above.

\sn
$\mathrm{(iv)}$ 
Moreover if we assume $X$ is noetherian, we have the following equivalent conditions.\\
\quad $\mathrm{(a)}$ $E^{\bullet}$ is pseudo-coherent.\\
\quad $\mathrm{(b)}$ $E^{\bullet}$ is cohomologically bounded above and all the cohomology sheaf $H^k(E^{\bullet})$ are coherent $\OX$-Modules.\\
In particular, a pseudo-coherent $\OX$-Module is coherent.

\sn
$\mathrm{(v)}$ 
The complex $E^{\bullet}$ is perfect if and only if 
$E^{\bullet}$ is pseudo-coherent and has locally finite Tor-amplitude.

\sn
$\mathrm{(vi)}$ 
Pseudo-coherence and perfection have $2$ out of $3$ properties. 
Namely, let $E^{\bullet}$, $F^{\bullet}$ and $G^{\bullet}$ 
be the three vertexes of a distinguished triangle in the derived category of 
$\Mod (X)$ and if two of these three vertexes are pseudo-coherent (resp. perfect) 
then the third vertex is also.

\sn
$\mathrm{(vii)}$ 
For any complexes of $\OOO_X$-Modules $F^{\bullet}$ and $G^{\bullet}$, 
$F^{\bullet}\oplus G^{\bullet}$ is pseudo-coherent (resp. perfect) 
if and only if $F^{\bullet}$ and $G^{\bullet}$ are. 

\sn
$\mathrm{(viii)}$
A strictly bounded complex of perfect $\OX$-Modules $E^{\bullet}$ is perfect.
\end{lem}
%
%
%

\begin{df}
\label{perf df}
(i)  For any $\OOO_X$-Module $\FFF$, 
we denote its {\it{support}} by 
$$
  \Supp \FFF:=\{x\in X;\FFF_x\ne 0\}.
$$

\sn
(ii) (\cite{Tho97}, 3.2)
For a complex of $\OX$-Modules $E^{\bullet}$, 
the {\it{cohomological support}} of $E^{\bullet}$ 
is the subspace $\Supph E^{\bullet} \subset X$ 
those points $x\in X$ at which the stalk complex of $\OXx$-module $E_x^{\bullet}$ 
is not acyclic. 

\sn
(iii) 
For any closed subset $Y$ of $X$, 
we denote by $\PerfXonY$ (resp. $\PerfqcXonY$, $\sPerf(\XonY)$) 
the complicial biWaldhausen category of 
globally finite Tor-amplitude perfect complexes 
(resp. globally finite Tor-amplitude perfect complexes 
of quasi-coherent $\OX$-Modules, 
strictly perfect complexes) 
whose cohomological support on $Y$. 
Here, the cofibrations are the degree-wise split monomorphisms, 
and the weak equivalences are the quasi-isomorphisms. 
Put 
$$
  \Perf^r(X):= \bigcup_{\Codim Y \geqq r} \PerfXonY.
$$
\end{df}
\begin{lem}
\label{supp nt} 
$\mathrm{(i)}$ 
 For any short exact sequence of $\OX$-Modules
$$
  0 \to \FFF \to \GGG \to \HHH \to 0,
$$
we have $\Supp \GGG=\Supp \FFF \cup \Supp \HHH$.

\sn
$\mathrm{(ii)}$ 
For a complex of $\OX$-Modules $E^{\bullet}$, 
we have $\Supph E^{\bullet}= \cup_{n\in\Z} \Supp \Homo^n(E^{\bullet})$.
\end{lem}

\begin{lem}
{\label{coherent case}}
For a coherent scheme $X$ and its closed set $Y$, 
the canonical inclusion functor $\PerfqcXonY \inj \PerfXonY$ induces 
an equivalence of categories 
between thir derived categories.
\end{lem}
\begin{proof}
The inverse functor of $\TTT(\PerfqcXonY) \to \TTT(\PerfXonY)$ is 
given by the coherator (Prop.~\ref{coherator prop}, (iii)).  
\end{proof}

\subsection{Divisorial schemes}

Since perfect and pseudo-coherent complexes are well-behavior on 
{\it divisorial}\/ schemes, 
we briefly review the definition and fundamental properties of divisorial schemes.

\begin{df}[\cite{BGI71}, II, 2.2.5; \cite{TT90}, Def.~2.1.1]
\label{divisorial df} 
A coherent scheme $X$ is said to be {\it{divisorial}} 
if it has {\it an ample family of line bundles}. 
That is it has a family of line bundles $\{\LLL_{\alpha}\}$ 
which satisfies the following condition 
(see \opcit for another equivalent conditions):\\
For any $f\in\Gamma(X,\LLL_{\alpha}^{\otimes n})$, 
we put the open set 
$$
  X_f:=\{x\in X|f(x)\ne 0\}.
$$
Then $\{X_f\}$ is a basis for the Zariski topology of $X$ 
where $n$ runs over all positive integer, 
$\LLL_{\alpha}$ runs over the family of line bundles and 
$f$ runs over all global sections of all of $\LLL_{\alpha}^{\otimes n}$.
\end{df}

\begin{ex} 
(i) A quasi-projective scheme over affine scheme is divisorial. 
So classical algebraic varieties are divisorial. 
Since every scheme is locally affine, every scheme is locally divisorial.

\sn
(ii) A separated regular noetherian scheme is divisorial.

\sn
(iii) (\cite{TT90}, Exerc.~8.6) Let $k$ be an field and 
$X$ an $\mathbb{A}^n_{k}$ with double origin. 
Then $X$ is regular noetherian but is not divisorial.
\end{ex}

\begin{lem}[\cite{DG61}, II, 5.5.8 and  \cite{BGI71}, II, 2.2.3.1]
 \label{fundamental result about invertible sheaves} 
For a line bundle $\LLL$ on $X$ and a section $f\in \Gamma(X,\LLL)$, the canonical open immersion $X_f \to X$ is affine.

\end{lem}

\begin{thm}[Global resolution theorem, \cite{BGI71}, II; \cite{TT90}, Prop.~2.3.1]
\label{global resol thm}
Let $X$ be a divisorial scheme. Then we have the following.

\sn
{\rm (i)} Any pseudo-coherent complex of quasi-coherent $\OX$-Modules 
is globally quasi-isomorphic to a strictly pseudo-coherent complex.

\sn
{\rm (ii)} Any perfect complex is isomorphic to 
a strictly perfect complex in $\DDD(\Mod(X))$.
\end{thm}

\subsection{Regular closed immersion}

There are several definitions of 
regular immersion (see \cite{DG67} and \cite{BGI71}, VII). 
Both definitions are equivalent if a total scheme is noetherian. 
We adopt the definition in \cite{BGI71} and 
for readers convenience, we briefly review 
the notation and fundamental properties of regular closed immersion. 

\begin{df}
Let $u:\LLL\to \OX$ be a morphism of $\OX$-Modules 
from an algebraic vector bundle $\cL$  to $\OX$. 
A {\it{Koszul complex}} associated to $u$ is 
the strictly perfect complex $\kos^{\bullet}(u)$ defined as follows: 
For $n>0$, we put 
\begin{align*}
  \kos^{-n}(u)(=\kos_n(u)) &:= \bigwedge^n \LLL, \quad \text{and} \\
  d_n(x_1\wedge\cdots\wedge x_n)& :=\sum^n_{r=1}{(-1)}^{r-1}u(x_r)x_1\wedge\cdots\wedge\widehat{x}_r\wedge\cdots\wedge x_n.
\end{align*}
\end{df}

\begin{df}[\cite{BGI71}, VII, 1.4]
\label{regular immersion def} 
(i) An $\OX$-Module homomorphism $u:\LLL \to \OOO_X$ 
from an algebraic vector bundle $\cL$ 
to $\OX$ is said to be {\it{regular}} 
if $\kos^{\bullet}(u)$ is a resolution of $\OX/\im u$.

\sn
(ii) An ideal sheaf $\III$ on $X$ is {\it{regular}} 
if locally on $X$,  
there is a regular map $u:\LLL \to \OOO_X$ such that $\im u=\III$. 
More precisely, this means that 
if there is an open covering $\{U_i\}_{i\in I}$ of $X$ and for each $i\in I$, 
there is a regular map $u_i:\LLL|_{U_i}\to \OOO_{U_i}$ 
such that $\im u_i=\III|_{U_i}$.

\sn
(iii) A closed immersion $Y\hookrightarrow X$ is said to be {\it{regular}} 
if the defining ideal of $Y$ is regular.
We put $\NNN_{X/Y}:=\III/\III^2$ and call it 
the {\it{conormal sheaf}} of the regular closed immersion.
\end{df}

\begin{lem}[\cite{DG67}]
\label{fund res reg clo} 
Let  $Y\hookrightarrow X$ be a regular closed immersion 
whose defining ideal is $\III$.

\sn
$\mathrm{(i)}$ The ideal sheaf $\III$ satisfies the following conditions:

\noindent
\quad $\mathrm{(a)}$ $\III$ is of finite type.\\
\quad $\mathrm{(b)}$ For each $n$, 
$\III^n/\III^{n+1}$ is a locally free $\OOO_X/\III$-Module of finite type.\\
\quad $\mathrm{(c)}$ A canonical map
$$
  \Sym_{\OOO_X/\III}(\NNN_{X/Y}) \to \Gr_{\III}(\OOO_X)
$$
is an isomorphism of $\OOO_X/\III$-Algebra. 
Here $\Sym_{\OOO_X/\III}(\NNN_{X/Y})$ is 
the symmetric algebra associated to $\NNN_{X/Y}$, 
$\Gr_{\III}(\OOO_X)=\bigoplus_{n\geqq 0} \III^n/\III^{n+1}$ is the graded algebra 
associated to an $\III$-adic filtration in $\OOO_X$ and 
the canonical map is defined by the universal property of symmetric algebra. 


\sn
$\mathrm{(ii)}$ If the scheme $X$ is noetherian, 
then $\III$ is regular in the sense of \cite{DG67}. 
That is, for any point $x\in X$ 
there is an open neighborhood $U$ of $x$, 
and a regular sequence $f_1,\ldots,f_r\in\Gamma(U,\III)$ which generates $\III|_{U}$.   
\end{lem}

\section{Weight on pseudo-coherent Modules}


\begin{df}
\label{weight def}
A pseudo-coherent $\OX$-Module  $\cF$  is {\it of weight} $r$ 
if it is of Tor-dimension $\leqq r$ and 
there is a regular closed immersion $Y \hookrightarrow X$ of codimension $r$ in $X$ 
such that the support of $\cF$ is in $Y$.
\end{df}
We denote by $\Wt^r(X)$ 
the category of pseudo-coherent $\OX$-Modules of weight $r$. 
For a regular closed immersion $Y \hookrightarrow X$ of codimension $r$,  
we denote by $\WtrXonY$ the category of pseudo coherent $\OX$-Modules 
of weight $r$ supported on the subspace $Y$.
Immediately, a pseudo-coherent $\OX$-Module of weight $0$ 
is just an algebraic vector bundle. 

\begin{lem} 
\label{weight lemma} 
The category $\WtrXonY$ is closed under extensions and 
direct summand in the abelian category $\ModX$. 
In particular, 
$\WtrXonY$ is an idempotent complete exact category.
\end{lem}
\begin{proof}
The assertion follows from Lemma \ref{fund fact of flat module}, (iii), 
Lemma \ref{supp nt}, (i), 
and Lemma \ref{fund fact for psuedo-coh and perf}, (vii).
\end{proof}

A pseudo-coherent $\OX$-Module $\cF$ of weight $r$ 
has globally finite Tor-amplitude. 
Thus it is perfect by Lemma~\ref{fund fact for psuedo-coh and perf}, (v) 
and we have an inclusion functor 
$\WtrXonY \hookrightarrow \Perf(\XonY)$.
%
Moreover we have the natural inclusion functor $\Ch^b(\WtrXonY) \inj \Perf(\XonY)$ by Lemma~\ref{fund fact for psuedo-coh and perf}, (viii).
Now, we state our main theorem.

\begin{thm} \label{main theorem}
Let $X$ be a divisorial scheme and 
$Y \inj X$ a regular closed immersion of codimension $r$. 
Then the inclusion $\Ch^b(\WtrXonY) \inj \PerfXonY$ 
induces a derived Morita equivalence.
\end{thm}

Now consider the inclusion functor 
$\WtrXonY \to \Ch^b(\WtrXonY)$ 
which sends 
$\cF$ in $\WtrXonY$ to the complex which is $\cF$ in degree $0$ 
and $0$ in other degrees. 
We denote by $K^S(\Ch^b(\WtrXonY); \qis)$ 
the $K$-theory spectrum of 
the Waldhausen category associated to $\Ch^b(\WtrXonY)$ whose 
weak equivalences are the quasi-isomorphisms. 
The inclusion above induces a homotopy equivalence 
$$
  K^S(\WtrXonY) \isoto K^S(\Ch^b(\WtrXonY);\qis) 
$$
by non-connected version 
of the Gillet-Waldhausen theorem in \cite{Sch04}. 
Therefore we get the following corollary.

\begin{cor}
In the notation above, we have the identities
\begin{gather*}
  K^S(\WtrXonY)\isoto K^S(X\on Y) \isoto K^B(X\on Y),\\
 HC(\WtrXonY)\isoto HC(\XonY).
\end{gather*}
\end{cor}

\begin{proof}
For the $K$-theory case, 
it is followed 
from the observation above 
and the Schlichting approximation theorem  
and the comparison theorem in \cite{Sch06}. 
For the cyclic homology case, 
it is followed from
the derived invariance by \cite{Kel99}.
\end{proof}

\section{Proof of the main theorem}

First we consider the following two categories. 
Let 
$\BBB$ be the category of perfect complexes in $\Ch^{-}(\WtrXonY)$ and 
$\CCC$ the category of perfect complexes of 
quasi-coherent $\OX$-Modules supported on $Y$. 
By Lemma~\ref{fund fact for psuedo-coh and perf}, 
the categories $\BBB$ and $\CCC$ are closed under extensions and 
direct summand in $\Ch(\Mod(X))$. 
Therefore, they are idempotent complete exact categories. 
Note that any perfect complex has globally finite Tor-amplitude 
on $X$ 
(Rem.~\ref{tor-amp rem} and Lem.~\ref{fund fact for psuedo-coh and perf}, (v)).
From Lemma~\ref{fund fact for psuedo-coh and perf}, (iii), 
we have the following natural exact inclusion functors
$$
\Ch^b(\WtrXonY) \onto{\alpha} 
  \BBB \onto{\beta} \CCC \onto{\gamma} \PerfXonY.
$$
We shall prove $\alpha$, $\beta$ and $\gamma$ 
induce category equivalences 
between their associated derived categories by using the following criterion.

\begin{lem}[\cite{TT90}, 1.9.7 and \cite{Tho93}] \label{Deri eq cri}
Let $i:\XXX \to \YYY$ be a fully faithful complicial exact functor 
between complicial biWaldhausen categories 
which closed under the formation of canonical homotopy pullbacks 
and pushouts and assume 
their weak equivalence classes are just quasi-isomorphism classes. 
If $i$ satisfies the condition $(\DE)$ or $(\DE)^{\op}$ below, 
then $i$ induces category equivalences between their derived categories.\\
\quad $(\DE)$  
For any object $Y$ in $\YYY$,  
there is an object $X$ in $\XXX$ and a weak equivalence $i(X) \to Y$.\\
\quad $(\DE)^{\op}$  
For any object $Y$ in $\YYY$, 
there is an object $X$ in $\XXX$ and a weak equivalence $Y \to i(X)$.
\end{lem}


%
We shall prove that $\alpha$ induces category equivalence between their derived categories. 
To do so first we review the following lemma.

\begin{lem}[\cite{BS01}, 2.6] \label{Bal-Sch lemma}
Let $\EEE$ be an idempotent complete exact categories and 
$f:X^{\bullet} \to Y^{\bullet}$ a quasi-isomorphism 
between bounded above complexes in $\Ch(\EEE)$.
Assume $X^{\bullet}$ or $Y^{\bullet}$ is strictly bounded. 
Say the other one as $Z^{\bullet}$. 
Then there is a sufficiently small $N$ 
such that $Z^{\bullet} \to \tau^{\geqq N}Z^{\bullet}$ is a quasi-isomorphism.
\end{lem}

\begin{lem} 
The inclusion $\alpha:\Ch^b(\WtrXonY)\inj \BBB$ satisfies 
the condition $(\DE)^{\op}$ in Lemma~\ref{Deri eq cri}. 
In particular, we have an equivalence of categories 
$$
  \TTT(\Ch^b(\WtrXonY)) \isoto \TTT(\BBB).
$$
\end{lem}
\begin{proof}
Let 
$\EEE$ be the category of pseudo-coherent $\OX$-Modules of Tor-dimension $\leqq r$. 
It is closed under extensions (Lem.~\ref{fund fact of flat module}, (iii)) 
and direct summand (Lem.~\ref{fund fact for psuedo-coh and perf}, (vii)). 
In particular, it is an idempotent complete exact category. 
We denote by 
$\DDD$ the category of perfect complexes in $\Ch^{-}(\EEE)$ whose 
cohomological support is in $Y$. 
Fix a complex $P^{\bullet}$ in $\BBB$. 
By the global resolution theorem (Th.~\ref{global resol thm}), 
$\Pb$ is quasi-isomorphic to a strict perfect complex. 
Since we have an inclusion $\sPerf(X\on Y)\subset\DDD$, 
$P^{\bullet}$ is quasi-isomorphic to a bounded complex in $\DDD$.
Now applying Lemma~\ref{Bal-Sch lemma} to $\EEE$, 
there exists an integer $N$ such that 
the canonical map $P^{\bullet} \to \tau^{\geqq N}P^{\bullet}$ 
is a quasi-isomorphism. 
Since $\Supp(\im d^{N-1})$ is in $Y$, 
$\tau^{\geqq N}P^{\bullet}$ is actually in $\Ch^b(\WtrXonY)$. 
The assertion follows from it.
\end{proof}

%

\begin{prop} \label{enough objects to resolve prop} 
The inclusion functor $\beta: \BBB\inj \CCC$ satisfies the condition $(\DE)$ 
in Lemma~\ref{Deri eq cri}. 
\end{prop}

To prove Proposition~\ref{enough objects to resolve prop}, 
we need the following lemmas.

\begin{lem}
\label{weight ex}
$\mathrm{(i)}$ Let $\cI$ be the definition ideal of $Y$. 
Then $\OX/\cI^p$ is of weight $r$ for any non-negative integer $p$.

\sn
$\mathrm{(ii)}$ 
Let $\cF$  be a pseudo-coherent $\OX$-Module of weight $r$ and 
$\cL$ an algebraic vector bundle. Then, 
$\LLL\otimes_{\OOO_X}\FFF$ is of weight $r$.
\end{lem}
\begin{proof}
(i) 
First we notice that $\OOO_X/\III$ is in $\WtrXonY$ by Koszul resolution. 
Next 
since $\III^n/\III^{n+1}$ is locally isomorphic to direct sum of $\OOO_X/\III$, 
we learn that $\III^n/\III^{n+1}$ is also in $\WtrXonY$ 
by Lemma~\ref{fund fact of flat module} (iv). 
Using Lemma~\ref{weight lemma} for 
$$
  0 \to \III^{n+1}/\III^{n+p} \to \III^n/\III^{n+p} \to \III^n/\III^{n+1} \to 0,
$$
the d\'evissage argument shows that 
$\III^n/\III^{n+p}$ is also in $\WtrXonY$ 
for any non-negative integer $n$ and positive integer $p$. 

\noindent 
(ii)
Since $\LLL$ is flat, 
we have an inequality $\Tordim(\LLL\otimes_{\OOO_X}\FFF)\leqq r$. 
We also have a formula 
$$
  \Supp \LLL\otimes_{\OOO_X}\FFF = \Supp \LLL \cap \Supp \FFF \subset Y.
$$
Therefore $\LLL\otimes_{\OOO_X}\FFF$ is of weight $r$. 
\end{proof}
\begin{lem}[\cite{TT90}, Lem.~1.9.5] 
\label{enough objects to resolve lemma}
Let $\AAA$ be an abelian category and 
$\DDD$ a full sub additive category of $\AAA$. 
Let $\CCC$ be a full subcategory of $\Ch(\AAA)$ satisfies the following conditions:

\noindent
{\rm (a)} $\CCC$ is closed under quasi-isomorphisms. 
That is, any complex quasi-isomorphic to an object in $\CCC$ is also in $\CCC$.\\
{\rm (b)} Every complex in $\CCC$ is cohomologically bounded above.\\
{\rm (c)} $\Ch^b(\DDD)$ is contained in $\CCC$.\\
{\rm (d)} $\CCC$ contains the mapping cone of any map 
from an object in $\Ch^b(\DDD)$ to an object in $\CCC$. \\
Finally, Suppose the following condition, 
so ``$\DDD$ has enough objects to resolve'':\\
{\rm (e)} 
For any integer $n$, any $C^{\bullet}$ in $\CCC$ such that 
$\Homo^i(C^{\bullet})=0$ for any $i\geqq n$ and any epimorphism in $\AAA$, 
$A \twoheadrightarrow \Homo^{n-1}(C^{\bullet})$, 
then there exists a $D$ in $\DDD$ and a morphism $D \to A$ 
such that the composite $D\twoheadrightarrow \Homo^{n-1}(C^{\bullet})$ is 
an epimorphism in $\AAA$.

Then, for any $D^{\bullet}$ in $\Ch^{-}(\DDD) \cap \CCC$, 
any $C^{\bullet}$ in $\CCC$, and any morphism $x:D^{\bullet} \to C^{\bullet}$, 
there exist a $E^{\bullet}$ in $\Ch^{-}(\DDD) \cap \CCC$, 
a degree-wise split monomorphism $a:D^{\bullet}\to E^{\bullet}$ and 
a quasi-isomorphism $y:E^{\bullet}\isoto C^{\bullet}$ such that $x=y\circ a$. 
Moreover if $x:D^{\bullet} \to C^{\bullet}$ is an $n$-quasi-isomorphism 
for some integer $n$, then one may choose $E^{\bullet}$ above so that 
$a^k:D^k \to E^k$ is an isomorphism for $k\geqq n$. 
\end{lem}

\begin{lem} \label{divisorial lemma} 
Let $X$ be a divisorial scheme 
whose ample family of line bundles is $\{\LLL_\alpha\}$ and 
$E^{\bullet}$ a perfect complex on $X$. 
Then there are line bundles $\LLL_{\alpha_k}$ in the ample family, 
integers $m_k$ and sections $f_k\in\Gamma(X,\LLL_{\alpha_k}^{\otimes m_k})$ 
$(1 \leqq k \leqq m)$ such that\\
{\rm (a)} For each $k$, $X_{f_k}$ is affine.\\
{\rm (b)} $\{X_{f_k}\}_{1 \leqq k \leqq m}$ is an open cover of $X$.\\
{\rm (c)} For each $k$, $E^{\bullet}|_{X_{f_k}}$ is quasi-isomorphic to a strictly perfect complex.
\end{lem}
\begin{proof}
Since $E^{\bullet}$ is perfect, 
we can take an affine open covering $\{U_i\}_{i\in I}$ of $X$ 
such that $E^{\bullet}|_{U_i}$ is quasi-isomorphic to 
a strictly perfect complex for each $i\in I$. 
Since $\{\LLL_\alpha\}$ is an ample family, 
for each $x\in $X, there are an $i_x\in I$, 
a line bundle $\LLL_{\alpha_x}$ in the ample family, 
an integer $m_x$ and a section $f_x\in\Gamma(X,\LLL_{\alpha_x}^{\otimes m_x})$ 
such that $x \in X_{f_x} \subset U_{i_x}$. Since $U_{i_x}$ is affine, 
$X_{f_x}$ is affine by Lemma~\ref{fundamental result about invertible sheaves}. 
Now $\{X_{f_x}\}_{x\in X}$ is an affine open covering of $X$ and 
has a finite sub covering by quasi-compactness of $X$. 
\end{proof}

\begin{lem}[\cite{TT90}, Lem.~1.9.4, (b)] \label{finiteness of cohomology} 
Let $E^{\bullet}$ be a strictly pseudo-coherent complex on $X$ 
such that $\Homo^{i}(E^{\bullet})=0$ for $i\geqq m$. 
Then $\Ker d^{m-1}$ is an algebraic vector bundle. 
In particular $\Homo^{m-1}(E^{\bullet})$ is of finite type.
\end{lem}

\begin{proof}[Proof of Prop.~\ref{enough objects to resolve prop}]
Let $\{\LLL_\alpha\}$ be an ample family of line bundles on $X$ and 
$\III$ the defining ideal of $Y$. 
We denote by $\DDD$ the additive category generated 
by all the $\LLL_{\alpha}^{\otimes m}\otimes_{\OOO_X}\OOO_X/\III^p$ 
with integer $m$ and positive integer $p$. 
By Lemma~\ref{weight ex}, $\DDD \subset \WtrXonY$. 
We intend to apply 
Lemma \ref{enough objects to resolve lemma} to 
$\AAA =\QcohXonY$ 
the category of quasi-coherent $\OX$-Modules whose support on $Y$. 
To do so, 
we have to check the assumptions in Lemma~\ref{enough objects to resolve lemma}. 
Only non-trivial assumption is 
\lq\lq having enough objects to resolve\rq\rq\ condition. 
Let $C^{\bullet}$ be a complex in $\CCC$ such that 
$\Homo^i(C^{\bullet})=0$ for $i\geqq n$, 
and $\FFF \twoheadrightarrow\Homo^{n-1}(C^{\bullet})$ an epimorphism in $\AAA$. 
By Lemma~\ref{divisorial lemma}, 
there are line bundles $\LLL_{\alpha_k}$ integers $m_k$ and their sections 
$f_k\in\Gamma(X,\LLL_{\alpha_k}^{\otimes m_k})$ ($1 \leqq k \leqq m$) 
such that they satisfy the following conditions.\\
(a) For each $k$, $X_{f_k}$ is affine.\\
(b) $\{X_{f_k}\}_{1 \leqq k \leqq m}$ is an open cover of $X$.\\
(c) For each $k$, $C^{\bullet}|_{X_{f_k}}$ is quasi-isomorphic to 
a strictly perfect complex.\\
Fix an integer $k$. 
Since $\Homo^{n-1}(C^{\bullet})|_{X_{f_k}}$ is of finite type 
by Lemma~\ref{finiteness of cohomology}, 
there is sub $\OOO_{X_{f_k}}$-Module of finite type $\GGG \subset \FFF|_{X_{f_k}}$ 
such that the composition 
$\GGG \hookrightarrow \FFF|_{X_{f_k}} \twoheadrightarrow \Homo^{n-1}(C^{\bullet})|_{X_{f_k}}$ is an epimorphism. 
Now since $\GGG$ and $\III|_{X_{f_k}}$ are $\OOO_{X_{f_k}}$-Modules of finite type 
(Lemma~\ref{fund res reg clo}, (i)), 
we have ${(\III|_{X_{f_k}})}^{p_k} \GGG =0$ for some $p_k$. 
Therefore $\GGG$ is considered as 
$\OX/\III^{p_k}|_{X_{f_k}}$-Module of finite type. 
Hence we have an epimorphism 
${(\OX/\III^{p_k}|_{X_{f_k}})}^{\oplus t_k} \twoheadrightarrow \GGG$. 
We have an $\OX$-Modules homomorphism 
${(\OX/\III^{p_k})}^{\oplus t_k} \to \FFF\otimes_{\OOO_X}\LLL_{\alpha_k}^{\otimes m_k s_k}$ for some integer $s_k$ 
(\cite{DG60}, 9.3.1 and \cite{DG64}, 1.7.5). 
Therefore considering the same argument for every $k$, we get a morphism
$$
\bigoplus_{k=1}^m (\OX/\III^{p_k} \otimes_{\OOO_X}\LLL_{\alpha_k}^{\otimes -m_k s_k})^{\oplus t_k} \to \FFF
$$
whose composition with $\FFF \twoheadrightarrow \Homo^{n-1}(C^{\bullet})$ 
is an epimorphism in $\QcohXonY$.
\end{proof}
%


%
Finally, we shall prove that $\gamma$ 
induces category equivalence between their derived categories. 
Now we consider the following exact inclusion functors:
$$
  \cC \onto{\gamma_1} \PerfqcXonY \onto{\gamma_2} \Perf(\XonY).
$$
Lemma~\ref{coherent case} 
assert that $\gamma_2$ induces 
a homotopy equivalence on spectra. 
Thus, it is enough to show that 
the inclusion functor 
$\gamma_1$ induces 
an equivalence of categories 
between their derived categories. 
More strongly we show the following:

\begin{lem} {\label{local cohomology lem}}
The local cohomological functor 
$$
  R\Gamma_Y = 
  \underset{\longrightarrow}{\lim}\mathcal{EXT}(\OX/\III^p, ?) :\TTT(\PerfqcXonY)\to\TTT(\cC)
$$
gives inverse functor of the inclusion functor $\gamma_1$. 
\end{lem}
\begin{proof}
Let us consider the functor
$$
\Gamma_Y:= 
\underset{\longrightarrow}{\lim}\HOM(\OX/\III^p, ?) :\Qcoh(X) \to \Qcoh(\XonY). 
$$
Since $\III$ is of finite type, for any $\OX$-Module $\MMM$ in $\Qcoh(\XonY)$, we have the identity 
\begin{equation}
\label{gamma preserve}
\Gamma_Y\MMM=\MMM. 
\end{equation}
This identity and the existence of the canonical natural transformation 
$\Gamma_Y \to \id$ imply that $\Gamma_Y$ is a right adjoint functor of 
the inclusion $\Qcoh(\XonY) \hookrightarrow \Qcoh(X)$. 
Therefore we learn that $\Qcoh(\XonY)$ has enough injective objects and 
for any complex $\Eb$ in $\cC$ such that each components are 
injective quasi-coherent $\OX$-Modules, 
we have the identity $R\Gamma_Y\Eb=\Eb$ by (\ref{gamma preserve}). 
Combining the obvious fact that $\gamma_1$ is fully faithful, 
we conclude that $R\Gamma_Y$ gives an inverse functor of $\gamma_1$.
\end{proof}

\section{Applications}

{\label{CM df}}
In this section, 
we assume that $A$ is the Cohen-Macaulay ring of Krull dimension $d$ and $X=\Spec A$. 
By the very definition, the ring $A$ satisfies the following condition 
(\cf \cite{Bou98}, \S 2.5, Prop.~7): 
For any ideal $J$ in $A$ such that its hight $\hight J=r$, 
there is an $A$-regular sequence $x_1,\ldots, x_r$ contained in $J$.

In this case, 
a coherent $A$-module of weight $d$ is 
just a module of finite length and finite projective dimension. 

\begin{prop}
{\label{wt is well-defined}}
For any integer $0\leqq r\leqq d$, 
$\Wtr(X)$ is closed under extensions in $\Mod(X)$. 
In particular $\Wtr(X)$ is an idempotent complete exact category in the natural way.
\end{prop}
\begin{proof}
Let us consider the short exact sequence
$$
  \cF \rightarrowtail \cG \twoheadrightarrow \HHH
$$
in $\Mod(X)$ such that $\cF$ and $\HHH$ are in $\Wtr(X)$. 
Then we learn that $\cG$ is of Tor-dimension $\leqq r$ and 
$\Codim \Supp \cG \geqq r$. 
Therefore 
there is an $A$-regular sequence $x_1,\ldots,x_r$ such that 
$\Supp \cG \subset V(x_1,\ldots,x_r)$. Hence we conclude that $\cG$ is in $\Wtr(X)$.
\end{proof}

\begin{thm}
{\label{main theorem 2}}
For any integer $0\leqq r\leqq d$, the canonical inclusion functor 
$\Ch^b(\Wtr(X)) \hookrightarrow \Perf^r(X)$ is a derived Morita equivalence.
\end{thm}

\begin{proof}
We can write the categories $\Ch^b(\Wtr(X))$ and $\Perf^r(X)$ as follows.
\begin{align*}
\Ch^b(\Wtr(X))&=\underset{Y \subset X}{\lim} \Ch^b(\WtrXonY),\\
\Perf^r(X)&=\underset{Y \subset X}{\lim} \Perf(\XonY), 
\end{align*}
where the limits taking over the regular closed immersion of 
codimension $\geqq r$. Hence we get the result by Theorem~\ref{main theorem} and 
continuity of functor $\cT$.
\end{proof}

\begin{cor}
{\label{main theorem 2 cor}}
For any integer $0\leqq r\leqq d$, 
we have the canonical homotopy equivalence of spectra and mixed complexes
\begin{gather*}
K^S(\Wtr(X)) \isoto K^S(\Perf^r(X)), \\
HC(\Wtr(X)) \isoto HC(\Perf^r(X)).
\end{gather*}
\end{cor}

\begin{proof}
Since both theories are derived invariant, 
the statement is just a corollary of 
Theorem \ref{main theorem 2}.
\end{proof}

Now moreover we assume 
that $A$ is local 
and let $\mm$ be its maximal ideal. 
Then since $A$ is Cohen-Macaulay, 
$Y:=V(\mm) \inj X$ is a regular closed immersion. 
Therefore by Theorem~\ref{main theorem}, 
we learn 
that $K^S(\XonY)$ is homotopy equivalent 
to $K^S(\Wt(\XonY))$. 
Now recall that Weibel's $K$-dimensional conjecture.

\begin{con}[$K$-dimensional conjecture]
{\label{Weibel conj}}
For any noetherian scheme $Z$ 
of finite Krull-dimension $n$, 
and integer $q > n$, 
we have $K_{-q}^B(Z)=0$.
\end{con}  

This conjecture is recently proved 
for schemes which is 
essentially of finite type 
over a field of characteristic $0$ \cite{CHSW08}. 
According to the paper \cite{Bal07}, 
if for any local ring $\OOO_{Z,z}$ of $Z$, 
we have $K_{-q}^B(\Spec \OOO_{Z,z}\!\on\!\overline{\{z\}})=0$ 
for $q > \dim \OOO_{Z,z}$, 
then the conjecture above is true for $Z$. 
Therefore for any Cohen-Macaulay scheme, 
the conjecture is reduced 
to vanishing of $K^S_{-q}(\Wt(\XonY))$ for $q> d$.


\vspace{1cm}

\noindent
 Toshiro Hiranouchi \\
{\tt hiranouchi@math.kyushu-u.ac.jp}

\vspace{0.5cm}

\noindent
 Satoshi Mochizuki \\
 {\tt mochi81@hotmail.com}

\end{document}